\documentclass{amsart}%
\usepackage{amsfonts}
\usepackage{amsmath}
\usepackage{amssymb}
\usepackage{graphicx}%
\providecommand{\U}[1]{\protect\rule{.1in}{.1in}}
\newtheorem{theorem}{Theorem}[section]

\newtheorem{claim}[theorem]{Claim}

\newtheorem{corollary}[theorem]{Corollary}

\newtheorem{example}[theorem]{Example}

\newtheorem{lemma}[theorem]{Lemma}

\newtheorem{problem}[theorem]{Problem}
\newtheorem{proposition}[theorem]{Proposition}

\begin{document}
\title{Continuum Nash Bargaining Solutions}
\author{Micah Warren}
\address{University of Oregon \\
Fenton Hall\\
University of Oregon\\
Eugene, OR 97403}
\email{micahw@uoregon.edu}
\thanks{The author's work was supported in part by the NSF via DMS-1161498}
\maketitle

\begin{abstract}
Nash`s classical bargaining solution suggests that $N$ players in a
non-cooperative bargaining situation should find a solution that maximizes the
product of each player's utility functions. We consider a special case:
Suppose that the players are chosen from a continuum distribution $\mu$ and
suppose they are to divide up a resource $\nu$ that is also on a continuum.
\ The utility to each player is determined by the exponential of a distance
type function. The maximization problem becomes an optimal transport type
problem, where the target density is the minimizer to the functional
\[
F(\beta)=H_{\nu}(\beta)+W^{2}(\mu,\beta)
\]
where $H_{\nu}(\beta)$ is the entropy and $W^{2}$ is the 2-Wasserstein
distance. This minimization problem is also solved in the
Jordan-Kinderlehrer-Otto scheme. Thanks to optimal transport theory, the
solution may be described by a potential that solves a fourth order nonlinear
elliptic PDE, similar to Abreu's equation. \ Using the PDE, we prove solutions
are smooth when the measures have smooth positive densities. \ 

\end{abstract}

\section{Introduction}

In the 1950's \cite{Nash} John Nash characterized a solution to the bargaining
problem that has since been central to the theory. \ Namely, the Nash
bargaining solution is the allocation that maximizes the product of the
utilities to each player, over the total space of possible allocations of a
surplus. \ The Nash bargaining solution is not only mathematically natural but
can be achieved by strategic approaches, see \cite{BRW}. \ $\ \ $For more on
the economic theory, see \cite{Muthoo}.

In this paper, we consider utility functions that are given as an integral of
a utility density function, namely
\begin{equation}
U_{i}\left(  \nu_{i}\right)  =\int_{Y}s(p_{i},y)d\nu_{i}(y).\label{U1}%
\end{equation}
Here $Y$ is a surplus space$,$ $X$ is a player space, $\nu_{i}$ is a measure
on the space $Y$ that determines the allocation of resources to player
$p_{i}\in X,$ and
\[
s:X\times Y\rightarrow%
\mathbb{R}
^{+}%
\]
is a utility density function. \ More details are found in section 2. \ 

We begin by extending an observation of Schumacher \cite{Schumacher} to
continua.\ \ For utility of the form
\[
s(x,y)=e^{-c(x,y)}%
\]
for a cost function $c,$ Nash bargaining solutions will be optimal transport
plans between a measure $\mu$ which is known a priori$,$ and some measure
$\beta$ which is determined by the solution. \ Thus, for the purpose of
maximizing the product of the utility functions $U_{i}$, we need only to
parameterize a space of measures on the resource space, or equivalently, a set
of Kantorivich potentials. \ This allows us to do analysis in the spirit of
Benamou-Carlier-M\'{e}rigot-Oudat \cite{BCMO}, in order to show that discrete
solutions have continuous limits.

We consider the particular case when the spaces $X,Y$ $\subset%
\mathbb{R}
^{n}$ are both convex compact domains, players are distributed according to a
probability measure $\mu$, and surplus is distributed according to a measure
$\nu,$ both of which are absolutely continuous with respect to Lebesgue
measure, and the utility density is given by
\[
s(x,y)=e^{-|x-y|^{2}/2}.
\]
Following the arguments in \cite{BCMO}, we find that the limit of solutions
can be described by a gradient mapping, whose potential satisfies a fourth
order PDE. \ \ The problem becomes equivalent to minimizing
\begin{equation}
G(\varphi)=\int_{X}\ln\left(  \frac{\det D^{2}\varphi(x)}{\nu(\nabla
\varphi(x))}\right)  d\mu(x)+\frac{1}{2}\int_{X}\left\Vert x-\nabla
\varphi\right\Vert ^{2}d\mu,\label{ELF}%
\end{equation}
which gives a fourth order nonlinear PDE similar to Abreu's equation
\cite{Donaldson}. \ In particular, the quantity $\det D^{2}\varphi(x)$
satisfies a second order elliptic equation. \ In Section 4 we show that $\det
D^{2}\varphi(x)$ is bounded and H\"{o}lder continuous; from here we can apply
Caffarelli's Schauder theory for Monge-Amp\`{e}re equations. \ Smoothness will
follow by the general Schauder theory, see Section 5. \ \ \ The statements
about the boundedness and H\"{o}lder continuity should hold more generally,
for minimizers of functionals of the form
\begin{equation}
\min_{\beta\in P(Y)}\left\{  W^{2}(\mu,\beta)+\int_{Y}f\left(  \frac{d\beta
}{d\nu}\right)  d\nu\right\}  \label{ELF2}%
\end{equation}
where $f$ \ satisfies some convexity and other conditions.

Our main result (stated roughly for now) is as follows

\begin{theorem}
Under appropriate conditions on the measures $\mu,\nu$ and for
$c(x,y)=\left\vert x-y\right\vert ^{2}/2,$ solutions to finite player Nash
bargaining solutions converge to a continuum Nash bargaining solution.
\end{theorem}

We also have a regularity result:

\begin{theorem}
Suppose that $X,Y$ $\subset%
\mathbb{R}
^{n}$ are smooth, bounded, convex domains. \ Suppose that $\mu$ and $\nu$ and
smooth measure densities on $X$ and $Y$ respectively, that are bounded and
bounded away from zero. \ \ \ Then, the minimizers of the functional
(\ref{ELF}) are smooth. \ \ \ 
\end{theorem}

Regularity for the Jordan-Kinderlehrer-Otto Scheme on the torus was shown in
\cite{Lee}. \ The proof of regularity in \cite{Lee} for functionals of the
form (\ref{ELF}) is for small time scales (which is sufficient to prove
regularity of the JKO scheme) and assumes some differentiability. \ An
$L^{\infty}$ estimate on the densities obtained via the JKO scheme for a
general form of a parabolic-elliptic Keller-Segel type system, was shown
recently in \cite{CS}.  We are able to obtain H\"{o}lder estimates directly on
the minimizing measures without any assumptions of differentiability or
smallness. \ Abreu's equation involves minimizing a functional of the form
(\ref{ELF}), where the gradient term is replaced by a integral involving the
potential $\varphi.$ \ \ Details on regularity for Abreu's equation can be
found in \cite{Le}.

\section{Setup}

Our goal in this section is to describe the bargaining problem in terms of
measures. \ \ \ \ \ 

Suppose that $Y$ is any topological space, and $\nu$ is a Borel probability
measure on $Y.$ \ We call the pair $\left(  Y,\nu\right)  $ the surplus.
\ \ When there are $N$ players vying for portions of the surplus, we can think
of the space of possible allocations of the surplus as $N$-tuples of
non-negative Borel measures $\left(  \nu_{1},...,\nu_{N}\right)  $, such that%
\[
\sum_{i=1}^{N}\nu_{i}\leq\nu.
\]
The function
\[
s:X\times Y\rightarrow%
\mathbb{R}
^{+}%
\]
is called a utility density function. \ The value $s(x,y)$ gives the utility
to player at $x\in X$ for a unit of $y\in Y.$ \ As we will be integrating the
density, it makes most sense to assume the utility is linear in terms of a
fixed resource $y$, in the sense that the marginal utility to a player at $x$
of a unit $y$ is determined only by the utility density function, and not by a
function of the same or other player's allocations. \ 

\begin{example}
\bigskip Suppose two emporers have set up capital cities at points $x_{1}$ and
$x_{2}$ in a region $\Omega\subset%
\mathbb{R}
^{2}$, and must negotiate how to split the region. \ Each emporer decides that
the value of any unit $y$ of land to their kingdom is given by $e^{-|x_{i}%
-y|^{2}/2}.$ \ The possible allocations of the region are measures $\nu
_{1},\nu_{2}$ such that
\[
\nu_{1}+\nu_{2}=dy_{|\Omega}.
\]
The utility function to each is
\[
U_{i}(\nu_{i})=\int_{Y}e^{-|x_{i}-y|^{2}/2}d\nu_{i}(y).
\]

\end{example}

Notice that in the $N$-player case, when the utility density is positive,
pareto-optimal solutions are a set of $N$ measures with $\sum_{i=1}^{N}\nu
_{i}=\nu.$ \ If $X=\left\{  p_{1},...,p_{n}\right\}  $ and $\pi\in P(X\times
Y)$ is a probability measure such that $\pi_{Y}=\nu$ ( here $\pi_{Y}=\left(
\text{Proj}_{Y}\right)  _{\#}$ $\pi,$ i.e. the right marginal) we may consider
the measures $\nu_{i}=\pi_{|\left\{  p_{i}\right\}  \times Y}.$ \ In this
case, we can define the utility to each player as follows%

\begin{equation}
U_{i}\left(  \pi\right)  =\int_{\left\{  p_{i}\right\}  \times Y}%
s(p_{i},y)d\pi.\label{U2}%
\end{equation}
The Nash bargaining solution is determined by maximizing the Nash product,
namely
\[
\mathcal{N}=\prod_{i=1}^{N}U_{i}.
\]
(Note that we are assuming the disagreement point is the $0$-allocation, that
is, all players gets nothing when they do not agree.)\ Equivalently, one can
maximize the logarithm
\[
\ln\mathcal{N}=\sum_{i=1}^{N}\ln U_{i}.
\]

With this in mind, for the case of $N$ players, we define the Nash bargaining
problem as follows:

\begin{problem}
Find a measure $\pi$ that maximizes the functional
\begin{equation}
F(\pi)=\ln N+\frac{1}{N}\sum_{i=1}^{N}\ln\int_{\left\{  p_{i}\right\}  \times
Y}s(p_{i},y)d\pi\label{NP1}%
\end{equation}
over the space of measures $\pi\in P(X\times Y)$, under the constraint%
\[
\pi_{Y}=\nu.
\]

\end{problem}

The term $\ln N$ and factor $1/N$ do not affect the arg max, however, they are
present for normalization reasons which will become clear when we attempt to
build a continuous solution. \ 

\bigskip

\subsection{\bigskip Reformulating the functional}

\bigskip Working with (\ref{NP1}):%
\begin{align*}
F(\pi)  &  =\ln N+\frac{1}{N}\sum_{i=1}^{N}\left[  \ln\left(  \pi\left(
\left\{  p_{i}\right\}  \times Y\right)  \right)  +\ln\left(  \frac{1}%
{\pi\left(  \left\{  p_{i}\right\}  \times Y\right)  }\int_{\left\{
p_{i}\right\}  \times Y}s(p_{i},y)d\pi\right)  \right] \\
&  =\ln N+\frac{1}{N}\sum_{i=1}^{N}\left[  \ln\left(  \frac{\pi\left(
\left\{  p_{i}\right\}  \times Y\right)  }{1/N}\right)  +\ln(1/N)+\ln\left(
\frac{1}{\pi\left(  \left\{  p_{i}\right\}  \times Y\right)  }\int_{\left\{
p_{i}\right\}  \times Y}s(p_{i},y)d\pi\right)  \right] \\
&  =\frac{1}{N}\sum_{i=1}^{N}\left[  \ln\left(  \frac{\pi\left(  \left\{
p_{i}\right\}  \times Y\right)  }{1/N}\right)  +\ln\left(  \frac{1}{\pi\left(
\left\{  p_{i}\right\}  \times Y\right)  }\int_{\left\{  p_{i}\right\}  \times
Y}s(p_{i},y)d\pi\right)  \right]  .
\end{align*}
At this point, we define a measure on $X=\left\{  p_{1},...,p_{N}\right\}  $
as
\[
\alpha=\pi_{X}%
\]
or equivalently,
\[
\alpha(p_{i})=\pi\left(  \left\{  p_{i}\right\}  \times Y\right)  .
\]
Defining
\begin{equation}
\mu_{N}=\frac{1}{N}\sum_{i=1}^{N}\delta_{p_{i}} \label{defMun}%
\end{equation}
we get
\begin{align}
F(\pi)  &  =\frac{1}{N}\sum_{i=1}^{N}\left[  \ln\left(  \frac{\alpha(p_{i}%
)}{\mu_{N}\left(  p_{i}\right)  }\right)  +\ln\left(  \frac{1}{\alpha(p_{i}%
)}\int_{\left\{  p_{i}\right\}  \times Y}s(p_{i},y)d\pi\right)  \right]
\label{F2}\\
&  =\int_{X}\ln\left(  \frac{d\alpha}{d\mu_{N}}\right)  \frac{d\mu_{N}%
}{d\alpha}d\alpha+\int_{X}A\mu_{N}\nonumber
\end{align}
where
\[
A(p_{i})=\ln\left(  \frac{1}{\alpha(p_{i})}\int_{\left\{  p_{i}\right\}
\times Y}s(p_{i},y)d\pi\right)  .
\]
Note that if $\alpha(p_{i})$ $=0$, then $\int\ln\left(  \frac{d\alpha}{d\mu
}\right)  d\mu=-\infty,$ in which case $A(p_{i})$ may be undefined, but we
agree that $F(\pi)=-\infty.$

\bigskip

\subsection{Solutions are Optimal Transport Plans}

(Cf. \cite[section 3]{Schumacher}.) Suppose that $\pi$ is a maximizer for $F,$
and $s$ is a positive, continuous function. \ Define%
\[
c(x,y)=-\ln s(x,y)\text{.}%
\]
Recall that a measure $\pi\in P(X\times Y)$, is a solution to the optimal
transportation problem pairing the measures $\mu$ and $\nu,$ when $\pi$
minimizes%
\[
\int_{X\times Y}c(x,y)d\pi
\]
under the constraint
\begin{align*}
\pi_{X} &  =\mu\\
\pi_{Y} &  =\nu.
\end{align*}
By Kantorovich duality, we have that, \cite[Theorem 5.10 (iii)]{Villani}
\[
\min_{\pi\in P(X\times Y)}\int c(x,y)d\pi=\max_{\varphi\in L^{1}(X,\alpha
_{N})}\int_{Y}\varphi^{c}(y)d\nu-\int_{X}\varphi(x)d\mu.
\]
where $\varphi^{c}$ is the "cost-transform" of $\varphi.$ \ \ This is a
general fact. \ For many cost functions of interest (for example
$c(x,y)=\left\vert x-y\right\vert ^{2}/2)$ the functions $\varphi^{c}$ and
$\varphi$ are determined uniquely up to a constant. \ 

\begin{proposition}
\label{three}The measure $\pi$ is an optimal plan, pairing $\alpha$ and $\nu.$
\end{proposition}

\begin{proof}
We proceed as in \cite[proof of Theorem 5.10, step 3 on page 65]{Villani}.
\ We will show that $\pi$ is $c$-cyclically monotone. \ 

Suppose that
\[
\left\{  (x_{1},y_{1}),...,(x_{n},y_{n})\right\}  \subset\text{Supp}(\pi).
\]
\ We would like to show that
\[
\sum_{i=1}^{N}c\left(  x_{i},y_{i}\right)  \leq\sum_{i=1}^{N}c\left(
x_{i},y_{i-1}\right)
\]
(using convention that $y_{0}=y_{N}$). \ Moving a bit of mass from $\left(
x_{1},y_{1}\right)  $ to $\left(  x_{2},y_{1}\right)  $ will preserve the
right marginal condition, but cannot increase the functional, by maximality.
Because $\left(  x_{1},y_{1}\right)  \in$ Supp$(\pi),$ there is some mass
available to move. \ \ \ Choose an arbitrarily small set $B_{\varepsilon}$
near $\left(  x_{1},y_{1}\right)  $ in $y$ and consider the competing family
of measures for $t\in\lbrack0,1]:$
\begin{align*}
\pi(t) &  =\left\{  \tilde{\nu}_{1}(t),\tilde{\nu}_{2}(t),\nu_{3},...,\nu
_{N}\right\}  \\
\tilde{\nu}_{1} &  =\nu_{1}-t\nu_{1}|_{B_{\varepsilon}}\\
\tilde{\nu}_{2} &  =\nu_{2}+t\nu_{1}|_{B_{\varepsilon}}.
\end{align*}
Clearly $\pi(t)$ is a path of admissible measures, so we can take a one-sided
derivative of (\ref{NP1}) with respect to $t:$%

\[
0\geq\frac{dF}{dt}|_{t=0}=\frac{1}{N}\left[  \frac{-\int_{\left\{
p_{1}\right\}  \times B_{\varepsilon}}s(x_{1},y)d\nu_{1}}{\int_{\left\{
p_{1}\right\}  \times Y}s(x_{1},y)d\nu_{1}}+\frac{\int_{\left\{
p_{1}\right\}  \times B_{\varepsilon}}s(x_{2},y)d\nu_{1}}{\int_{\left\{
p_{2}\right\}  \times Y}s(x_{2},y)d\nu_{2}}\right]  .
\]
That is%
\[
0\geq\left[  \frac{-\frac{1}{\nu_{1}\left(  B_{\varepsilon}\right)  }%
\int_{\left\{  p_{1}\right\}  \times B_{\varepsilon}}s(x_{1},y_{1})d\nu_{1}%
}{\int_{\left\{  p_{1}\right\}  \times Y}s(x_{1},y)d\nu_{1}}+\frac{\frac
{1}{\nu_{1}\left(  B_{\varepsilon}\right)  }\int_{\left\{  p_{1}\right\}
\times B_{\varepsilon}}s(x_{2},y_{1})d\nu_{1}}{\int_{\left\{  p_{2}\right\}
\times Y}s(x_{2},y)d\nu_{2}}\right]  .
\]
Choosing $B_{\varepsilon}$ small, since $s$ is continuous, the average values
in the numerator must converge to point values, and in the limit we see \
\begin{equation}
\frac{s(x_{1},y_{1})}{\int_{\left\{  p_{1}\right\}  \times Y}s(x_{1}%
,y)d\nu_{1}}\geq\frac{s(x_{2},y_{1})}{\int_{\left\{  p_{2}\right\}  \times
Y}s(x_{2},y)d\nu_{2}}. \label{fla}%
\end{equation}
Then we have
\[
\ln s(x_{1},y_{1})-\ln s(x_{2},y_{1})\geq\ln\int_{\left\{  p_{1}\right\}
\times Y}s(x_{1},y)d\nu_{1}-\ln\int_{\left\{  p_{2}\right\}  \times Y}%
s(x_{2},y)d\nu_{2}%
\]
or
\[
c(x_{1},y_{1})-c(x_{2},y_{1})\leq\ln\kappa_{1}-\ln\kappa_{2}.
\]
By relabeling,%
\begin{align*}
c(x_{2},y_{2})-c(x_{3},y_{2})  &  \leq\ln\kappa_{2}-\ln\kappa_{3},\\
&  ...\\
c(x_{N},y_{N})-c(x_{1},y_{N})  &  \leq\ln\kappa_{N}-\ln\kappa_{1}.
\end{align*}
Summing, we have
\[
\sum_{i=1}^{N}c(x_{i},y_{i})-\sum_{i=1}^{N}c(x_{i},y_{i-1})\leq0.
\]

\end{proof}

\begin{corollary}
\bigskip Suppose that
\[
s(x,y)=e^{-|x-y|^{2}/2}%
\]
and $X,Y\subset%
\mathbb{R}
^{n}.$ \ Then the solution maximizing (\ref{NP1}) can be described as%
\[
\nu_{i}=\nu_{|E_{i}}%
\]
where each $E_{i}$ is a Laguerre cell, defined as the subgradient of a convex
function $\varphi$ at point $x_{i}.$
\end{corollary}

We won't prove this directly, but refer the reader to \cite[Chapter
5]{Villani} and \cite[Section 2]{BCMO}. \ We include a condensed recap of
\cite[Section 2]{BCMO} which will help our discussion moving forward: \ 

Suppose $P\subset%
\mathbb{R}
^{n}$ is a finite set of points, and $Y\subset%
\mathbb{R}
^{n}.$ \ \ Given a function $\varphi$ defined on $P,$ define
\[
\varphi_{\mathcal{K}_{Y}}=\max\left\{  \psi\in\mathcal{K}_{Y};\psi_{|_{P}}%
\leq\varphi_{|_{P}}\right\}
\]
where
\[
\mathcal{K}_{Y}=\left\{  \psi^{\ast};\psi:Y\rightarrow\bar{%
\mathbb{R}%
}\right\}
\]
is the set of functions which are Legendre-Fenchel transforms of functions on
$Y,$ which are necessarily convex. \ \ Define
\[
\mathcal{K}_{Y}(P)=\left\{  \varphi:P\rightarrow%
\mathbb{R}
;\varphi=\varphi_{\mathcal{K}_{Y}}|_{P}\right\}
\]
and
\[
Lag_{P}^{\varphi}(p):=\left\{  y\in%
\mathbb{R}
^{n};\forall q\in P,\varphi(q)\geq\varphi(p)+\langle q-p|y\rangle\right\}  .
\]
For normalization, we define
\[
\mathcal{K}_{Y}(P)_{0}=\left\{  \varphi\in\mathcal{K}_{Y}(P);\min
\varphi=0\right\}  .
\]

\begin{lemma}
\cite[Lemma 2.2]{BCMO} Let $P$ be a finite point set. A\ function $\varphi$ on
$P$ belongs \ to $\mathcal{K}_{Y}(P)$ if and only if for every $p$ in $P,$ the
intersection $Lag_{P}^{\varphi}(p)\cap Y$ is non-empty. \ \ Moreover, if this
is the case, then
\[
\partial\varphi_{\mathcal{K}}(p)=Lag_{P}^{\varphi}(p)\cap Y.
\]

\end{lemma}

\subsection{Reformulating the problem again.}

Now that we have established that the solution must be an optimal transport
plan, we may formulate the problem over the space of optimal transport plans.
\ \ First, \ we need

\begin{lemma}
A\ maximizing solution to (\ref{NP1}) exists, and is unique. \ 
\end{lemma}

\begin{proof}
The constraint $\pi_{Y}=\nu$ is linear on the set of nonnegative probability
measures, which form a compact convex set. For each $i$, the functional
\[
f_{i}=\int_{\left\{  p_{i}\right\}  \times Y}s(p_{i},y)d\pi
\]
is clearly linear in the measure $\pi.$ \ It follows that the sum of
logarithms is strictly concave. \ \ A\ concave function on a convex compact
set achieves its maximum value, which is unique by strict concavity. \ \ \ 
\end{proof}

It is clear that the maximizing measure must be supported on $P\times Y:$
\ Moving any mass from a point not in $P$ to a point on $P$ will increase the
value of $F.$ \ By Kantorovich duality, the optimal transport plans with
target $\nu$ can be parameterized by convex functions on the set $P,$ which is
just a set of $N$ values. \ The functional can be expressed as follows. \
\begin{equation}
\tilde{F}_{N}(\varphi)=\frac{1}{N}\sum_{i=1}^{N}\left[  \ln\left(  \frac
{\nu(E_{i})}{\mu_{N}\left(  p_{i}\right)  }\right)  +\ln\left(  \frac{1}%
{\nu(E_{i})}\int_{E_{i}}s(p_{i},y)d\nu\right)  \right]  \label{p3}%
\end{equation}
where
\[
E_{i}=Lag_{P}^{\varphi}(p_{i}).
\]

We now offer a second formulation of the problem.

\begin{problem}
(Nash bargaining problem, Version 2) Maximize \ (\ref{p3}) over the space of
functions $\mathcal{K}_{Y}(P)_{0}.$
\end{problem}

\bigskip

\subsection{Absolutely continuous pushforward measures}

Given a potential function $\varphi\in\mathcal{K}_{Y}(P)_{0}$, we have a
Laguerre decomposition of $Y.$ \ The subgradient map is set-valued, so one
cannot define the pushforward directly. \ However, following \cite{BCMO}, if
we have chosen a background measure $\nu,$ we can \textquotedblleft
average\textquotedblright\ over the Laguerre cell to define an absolutely
continuous pushforward measure. \ 

Given any decomposition of $Y$ into cells $E_{i}$, each with positive measure,
define the following probability measure on $Y:$ \ For measurable $Z\subset
Y,$
\begin{equation}
\beta_{\varphi}(Z)=\frac{1}{N}\sum_{i=1}^{N}\frac{\nu(E_{i}\cap Z)}{\nu
(E_{i})}.\label{acm}%
\end{equation}
In the measure $\pi,$ which is optimal between $\alpha$ and $\nu,$ each cell
$\left\{  p_{i}\right\}  \times E_{i}$ has measure $\nu(E_{i})$. \ By
multiplying each piece by
\[
\frac{1}{N}/\nu(E_{i})
\]
we get a new measure, $\pi^{\prime}$, which has the same support as $\pi,$ but
now has marginals%
\begin{align*}
\left(  \pi^{\prime}\right)  _{X} &  =\mu_{N}\\
\left(  \pi^{\prime}\right)  _{Y} &  =\beta_{\varphi}.
\end{align*}
This means that the set-valued mapping induced by the subgradient of $\varphi$
is an optimal transport mapping not only between $\alpha$ and $\nu\,,$ but
also between $\mu_{N}$ and $\beta_{\varphi}$.

\section{Limits of finite source solutions}

Our next goal is to take a limit of solutions, when the set of points is drawn
from a distribution $\mu$ and the number of points becomes infinite. \ \ In
the following we assume that $X,Y\subset%
\mathbb{R}
^{n}$ are both bounded convex regions, and that\ $\mu,\nu$ are probability
measures on $X,Y$ respectively, each of which are absolutely continuous with
respect to Lebesgue measure, with densities bounded away from zero. \ For each
$N,$ choose a set of points $P_{N}$ from $X\,$\ and define the measure
\[
\mu_{N}=\frac{1}{N}\sum_{i=1}^{N}\delta_{p_{i}}.
\]
In the sequel we will assume that $\mu_{N}\rightarrow\mu$ weakly on $X.$ \ By
arguments in the previous section, for each $N$ we have
\begin{align}
\pi_{N} &  \in P(X\times Y)\label{maxers}\\
\varphi_{N} &  \in\mathcal{K}_{Y}(P)_{0}\nonumber\\
\alpha_{N} &  \in P(X)\nonumber\\
\beta_{\varphi_{N}} &  \in P(Y)\nonumber
\end{align}
each of which are unique and can be used to identify the solution. \ \ We will
see that when $\mu,\nu$ are nice measures, all four of the objects
(\ref{maxers}) converge to limiting objects, respectively, each of which
uniquely defines a solution on the continuous space $X\times Y.$

\bigskip Inspecting (\ref{F2}), note the following. \ The first term is the
relative entropy of a known measure $\mu$ with respect to a measure $\alpha,$
to be determined. \ The second term is the integral of the natural log of the
average of an exponential function, over a small cell. \ As the size of the
cells becomes smaller, one expects the average to recover the value, and the
second term will become the negative total cost of the mass transport plan
between $\mu$ and its image $\beta$ under the mapping, which is also to be
determined. \ \ \ Under changes of measures
\begin{align*}
\mu &  \rightarrow\beta\\
\alpha &  \rightarrow\nu
\end{align*}
the first term becomes the negative relative entropy of $\beta$ with respect
to the measure $\nu.$ \ \ Thus we may formulate the problem by trying to
minimize the following function over the set of probability measures on $Y$ .
\
\begin{align}
\hat{F} &  :P(Y)\rightarrow%
\mathbb{R}
^{-}\nonumber\\
\hat{F}(\cdot) &  =-H_{v}(\cdot)-W^{2}(\mu,\cdot).\label{Fhatdef}%
\end{align}

The concave functional $\hat{F}$ will have a unique maximizer on $P(Y).$
\ Define%
\[
\hat{\beta}=\arg\max\hat{F}(\cdot)
\]
and choose $\hat{\varphi}$ such that
\[
\hat{\beta}=\left(  \nabla\hat{\varphi}\right)  _{\#}\mu.
\]

Our main theorem is the following. \ 

\begin{theorem}
\bigskip Let $\pi_{N}$ be a sequence of maximizers to (\ref{NP1}), and let
$\beta_{N}$ be the associated absolutely continuous right marginals
(\ref{acm}), \ \ Then
\[
\hat{F}(\beta_{N})\rightarrow\hat{F}(\hat{\beta}).
\]
In particular, because $\hat{F}$ is strictly concave,
\[
\beta_{N}\rightarrow\hat{\beta}.
\]

\end{theorem}

In this case, choose $\hat{\varphi}$ such that
\[
\hat{\beta}=\left(  \nabla\hat{\varphi}\right)  _{\#}\mu.
\]
We define $\hat{\varphi}$ as the potential solution, $\nabla\hat{\varphi}$ as
the map solution, and the measure $\pi=\left(  I\times\nabla\hat{\varphi
}\right)  _{\#}\mu\in P(X\times Y)$ as the measure solution of the Nash
bargaining problem. \ 

\subsection{Outline of Proof.}

We outline four steps, and then combine these in step 5 to get the proof.
\ The detailed proof of steps 1-4 will appear in section 4.

Note that this result almost follows from \cite[Theorem 4.1]{BCMO}. \ Instead
of a pure Wasserstein term, however, in our problem we have a term that
\textit{should }converge to a Wasserstein term, provided the measures
concentrate on the graph of a map. \ \ So we must justify that indeed the
maximimizers of our finite problem are indeed concentrating on the graph of
map. \ \ We reproduce some, but not all of the proof. \ \ 

\bigskip

Step 1. \ \ The functional $\hat{F}$ has maximizer, $\hat{\beta}.$ \ The
density $\frac{d\hat{\beta}}{d\nu}$ is bounded and H\"{o}lder continuous.

\bigskip

Step 2. \ There is a sequence of optimal transportation plans pairing $\mu
_{N}$ with $\hat{\beta}.$ \ To these we can associate a function $\hat
{\varphi}_{N}$ and an absolutely continuous measure $\hat{\beta}_{N.}$ \ Then
$\hat{\beta}_{N}\rightarrow\hat{\beta}$ and $\hat{F}(\hat{\beta}%
_{N})\rightarrow F(\hat{\beta}).$ \ The associated Laguerre cells have
diameters which go uniformly to $0.$

\bigskip

Step 3. \ For the maximizers $\varphi_{N}$ of the finite problem, we can also
associate an absolutely continuous measure $\beta_{N}$ (\ref{betaN}). \ The
limits of both exist, and the diameters of the associated Laguerre cells go
uniformly to $0.$

\bigskip

Step 4. \ Because the diameters of the Laguerre cells go to zero, the second
term $\ $in the $\tilde{F}_{N}$ functional (\ref{p3})  (defined on
$\varphi_{N}$ or $\hat{\varphi}_{N})$ converges to the second term in in the
$\hat{F}$ functional (defined on the associated $\beta_{N}).$ That is
\begin{equation}
\left\vert \frac{1}{N}\sum_{i=1}^{N}\ln\left(  \frac{1}{\nu(E_{i})}\int%
_{E_{i}}s(p_{i},y)d\nu(y)\right)  -W^{2}(\mu_{N},\beta_{N})\right\vert
\rightarrow0.\label{fagree}%
\end{equation}
It follows that
\begin{align}
\left\vert \hat{F}(\beta_{N})-\tilde{F}_{N}(\varphi_{N})\right\vert  &
\rightarrow0,\label{bbb}\\
\left\vert \hat{F}(\hat{\beta}_{N})-\tilde{F}_{N}(\hat{\varphi}_{N}%
)\right\vert  &  \rightarrow0.\label{bbb2}%
\end{align}
Step 5. \ \textbf{Proof of Theorem. } Choose $\varepsilon>0:$ \ Because we
have chosen $\varphi_{N}$ as a maximizer for (\ref{p3}) we have
\begin{equation}
\tilde{F}_{N}(\varphi_{N})\geq\tilde{F}_{N}(\hat{\varphi}_{N}).\label{maxprop}%
\end{equation}
Thus for $N$ large depending on $\varepsilon$,
\begin{align*}
\hat{F}\left(  \beta_{N}\right)   &  \geq\tilde{F}_{N}(\varphi_{N}%
)-\varepsilon\text{ by (\ref{bbb}) }\\
&  \geq\tilde{F}_{N}(\hat{\varphi}_{N})-\varepsilon\text{, by (\ref{maxprop})
}\\
&  \geq\hat{F}\left(  \hat{\beta}_{N}\right)  -2\varepsilon,\text{by
(\ref{bbb2}) }\\
&  \geq\hat{F}\left(  \hat{\beta}\right)  -3\varepsilon,\text{ by step 2.}%
\end{align*}
Thus,
\[
\lim_{N\rightarrow\infty}\hat{F}\left(  \beta_{N}\right)  \geq\hat{F}\left(
\hat{\beta}\right)  .
\]
But $\hat{F}\left(  \hat{\beta}\right)  $ is the maximum. \ By strict
concavity, it follows that $\beta_{N}\rightarrow\hat{\beta}.$ \ Thus
$\beta=\hat{\beta}.$

\section{Proof details}

\ Unless otherwise specified, $c$ will refer to an arbitrary cost function
\[
c:X\times Y\rightarrow%
\mathbb{R}
\]
and $W_{c}$ will refer to the Wasserstein distance, given the cost $c.$
$\ \ $The cost $c$ is assumed to have global bounds and global Lipschitz
bounds. \ \ \ \ The results of Step 1 should hold for general cost $c,$
whereas the results in Step 2 and 3 rely on Caffarelli's regularity theory for
$c=\left\vert x-y\right\vert ^{2}/2.$ \ 

Although we are most interested in the case $g(u)=\ln u,$ in Step 1 we
consider a slight generalization to functionals of the form
\[
F(\beta)=\int_{Y}g(\beta)\frac{d\beta}{d\nu}d\nu+W_{c}(\mu,\beta).
\]

\subsection{Step 1}

First an easy observation. \ 

\begin{claim}
$\label{costclaim}$Fixing $\mu,$ suppose that $\beta_{1},\beta_{2}$ are
probability measures, and that there is an open set $O$ such that $\beta
_{1}(E)=\beta_{2}(E)$ for all measurable $E\subset O,$ and $\beta_{1}(O)=$
$1-\eta.$ Then
\[
\left\vert W_{c}(\mu,\beta_{1})-W_{c}(\mu,\beta_{2})\right\vert \leq\left\Vert
c\right\Vert _{Lip\left(  X\times Y\right)  }\text{diam(}O)\eta,
\]

\end{claim}

\begin{proof}
\ Let $\pi$ be an optimal pairing for $\mu,\beta_{1}.$ \ We obtain a candidate
transport plan for pairing $\mu$ with $\beta_{2}$ by shifting mass vertically
along the set of size $\eta$ so that the right marginal is $\beta_{2}.$ \ The
cost of this plan can increase at most by $\left\Vert c\right\Vert _{Lip}%
$diam($Y)\eta.$
\end{proof}

\begin{lemma}
\label{linfbound}Suppose that $\hat{\beta}$ is the minimizer of a functional
of the form
\[
F(\beta)=\int_{Y}g(\beta)\frac{d\beta}{d\nu}d\nu+W_{c}(\mu,\beta)
\]
where $g$ is strictly increasing function,unbounded above, and
\[
l_{0}=\max_{s\in\lbrack0,1]}\left[  (s+1)g(s+1)-sg(s)\right]  <\infty.
\]
Then, the density $\frac{d\hat{\beta}}{d\nu}$ is bounded above, that is
\[
\hat{\beta}(E)\leq g^{-1}\left[  l_{0}+\text{osc}\left(  c\right)  \right]
\nu(E)
\]
for all measurable $E\subset Y$ . \ 
\end{lemma}

Note that in the case we are interested, $g(u)=\ln u.$ \ 

\begin{proof}
\ First, by evaluating the functional $F$ on the measure $\nu,$ we see
\[
F(\nu)=\int g(1)d\nu+W_{c}(\mu,\nu)<+\infty.
\]
It follows that the minimum $F(\hat{\beta})$ is also finite, in particular
$\frac{d\hat{\beta}}{d\nu}$ cannot be infinite on a set of positive $\beta$
measure, so $\frac{d\hat{\beta}}{d\nu}$ exists a.e. $d\beta$ and hence a.e.
$d\nu.$ \ We can decompose $Y$ (a.e $\nu)$ into a countable union:
\begin{align*}
Y &  =%
{\displaystyle\bigcup\limits_{M=0}^{\infty}}
E_{M}\\
E_{M} &  =\left\{  y:\frac{d\hat{\beta}}{d\nu}(y)\in\lbrack M,M+1)\right\}  .
\end{align*}
\ Choose an $M$ large, and assume that $\delta=\hat{\beta}(E_{M})>0$. \ Let
$\varepsilon=\nu(E_{M}).$ \ It follows that $\varepsilon\in\left[
\frac{\delta}{M+1},\frac{\delta}{M}\right]  .$ \ Choose a subset $\tilde
{E}\subset E_{0}$ such that $\nu(\tilde{E})=\delta$ \ (Such a set exists by a
Markov's inequality argument for any $M>1$.)\ \ By Jensen's Inequality (the
function $ug(u)$ is convex)
\begin{align*}
\int_{E_{M}}g\left(  \frac{d\hat{\beta}}{d\nu}\right)  \frac{d\hat{\beta}%
}{d\nu}d\nu &  =\varepsilon\int_{E_{M}}g\left(  \frac{d\hat{\beta}}{d\nu
}\right)  \frac{d\hat{\beta}}{d\nu}\frac{d\nu}{\varepsilon}\\
&  \geq\varepsilon g\left(  \int\frac{d\hat{\beta}}{d\nu}\frac{d\nu
}{\varepsilon}\right)  \int\frac{d\hat{\beta}}{d\nu}\frac{d\nu}{\varepsilon}\\
&  =\varepsilon g\left(  \frac{\delta}{\varepsilon}\right)  \frac{\delta
}{\varepsilon}\\
&  =\delta g\left(  \frac{\delta}{\varepsilon}\right)  .
\end{align*}
Now we create a competing measure, by moving the mass from $E_{M}$ to
$\tilde{E}.$ \ That is,
\[
\frac{d\beta^{\ast}}{d\nu}=%
\begin{cases}
0 & \text{on }E_{M}\\
\frac{d\hat{\beta}}{d\nu}+\chi_{\tilde{E}} & \text{otherwise}%
\end{cases}
\text{ .}%
\]
This gives another probability measure $\beta^{\ast}.$ \ Now $\beta^{\ast
}=\hat{\beta}$ except on these two small sets, so we have
\begin{align*}
&  \int_{Y}g\left(  \frac{d\hat{\beta}}{d\nu}\right)  \frac{d\hat{\beta}}%
{d\nu}d\nu-\int_{Y}g\left(  \frac{d\beta^{\ast}}{d\nu}\right)  \frac
{d\beta^{\ast}}{d\nu}d\nu\\
&  =\int_{E_{M}}g\left(  \frac{d\hat{\beta}}{d\nu}\right)  \frac{d\hat{\beta}%
}{d\nu}d\nu-\int_{E_{M}}g\left(  \frac{d\beta^{\ast}}{d\nu}\right)
\frac{d\beta^{\ast}}{d\nu}d\nu+\int_{\tilde{E}}g\left(  \frac{d\hat{\beta}%
}{d\nu}\right)  \frac{d\hat{\beta}}{d\nu}d\nu-\int_{\tilde{E}}g\left(
\frac{d\beta^{\ast}}{d\nu}\right)  \frac{d\beta^{\ast}}{d\nu}d\nu\\
&  =\int_{E_{M}}g\left(  \frac{d\hat{\beta}}{d\nu}\right)  \frac{d\hat{\beta}%
}{d\nu}d\nu-0+\int_{\tilde{E}}g\left(  \frac{d\hat{\beta}}{d\nu}\right)
\frac{d\hat{\beta}}{d\nu}d\nu-\int_{\tilde{E}}g\left(  \frac{d\hat{\beta}%
}{d\nu}+1\right)  \left(  \frac{d\hat{\beta}}{d\nu}+1\right)  d\nu\\
&  \geq\delta g\left(  \frac{\delta}{\varepsilon}\right)  -\int_{\tilde{E}%
}\left[  g\left(  \frac{d\hat{\beta}}{d\nu}+1\right)  \left(  \frac
{d\hat{\beta}}{d\nu}+1\right)  -\int_{\tilde{E}}g\left(  \frac{d\hat{\beta}%
}{d\nu}\right)  \frac{d\hat{\beta}}{d\nu}\right]  d\nu\\
&  \geq\delta g\left(  \frac{\delta}{\varepsilon}\right)  -\int_{\tilde{E}%
}l_{0}d\nu=\delta g\left(  \frac{\delta}{\varepsilon}\right)  -\delta l_{0}.
\end{align*}
That is
\begin{align*}
&  \int_{Y}g(\hat{\beta})\frac{d\hat{\beta}}{d\nu}d\nu+W_{c}(\mu,\hat{\beta
})-\left(  \int_{Y}g(\beta^{\ast})\frac{d\beta^{\ast}}{d\nu}d\nu+W_{c}%
(\mu,\beta^{\ast})\right)  \\
&  \geq\delta\left[  g\left(  \frac{\delta}{\varepsilon}\right)
-l_{0}\right]  -\delta\text{osc}\left(  c\right)
\end{align*}
using Claim \ref{costclaim}. Now using the minimizing property of $\hat{\beta
},$ we have
\[
g\left(  \frac{\delta}{\varepsilon}\right)  \leq l_{0}+\text{osc}\left(
c\right)  .
\]
We conclude%
\[
M\leq g^{-1}\left[  l_{0}+\text{osc}\left(  c\right)  \right]  .
\]
In particular, for larger values of $M,$ $\beta(E_{m})=0.$
\end{proof}

\begin{lemma}
$\label{holder}$Suppose that $\hat{\beta}$ is as in the previous Lemma and in
addition that
\[
\frac{d^{2}}{du^{2}}\left(  ug(u)\right)  >0.
\]
\ Then the density $\frac{d\hat{\beta}}{d\nu}$ is H\"{o}lder $1/2$ continuous. \ 
\end{lemma}

\begin{proof}
Choose any point $y$ in the interior of $Y.$ \ \ We claim that given
$\delta>0,$ there is an $\iota>0$ to be determined, such that for
\begin{equation}
\varepsilon=\iota\delta^{2}\label{oass}%
\end{equation}
we have
\[
\text{osc}_{B_{\varepsilon}(y)}\left(  \frac{d\hat{\beta}}{d\nu}\right)
<4\delta.
\]
Assume not. \ \ Suppose that osc$_{B_{\varepsilon}(y)}\left(  \frac
{d\hat{\beta}}{d\nu}\right)  \geq4\delta.$ \ Let
\begin{align*}
m_{+} &  =\text{ess}\sup_{B_{\varepsilon}}\frac{d\hat{\beta}}{d\nu}\\
m_{-} &  =\text{ess~}\inf_{B_{\varepsilon}}\frac{d\hat{\beta}}{d\nu}\\
L &  =\frac{m_{+}+m_{-}}{2}.
\end{align*}
It follows that
\begin{align*}
\nu\left(  \left\{  \frac{d\hat{\beta}}{d\nu}<L-\delta\right\}  \cap
B_{\varepsilon}(y)\right)   &  >0\\
\nu\left(  \left\{  \frac{d\hat{\beta}}{d\nu}>L+\delta\right\}  \cap
B_{\varepsilon}(y)\right)   &  >0.
\end{align*}
Choose small $\eta>0$ and sets $E_{-},E_{+}$ such that $\nu(E_{-})=\nu
(E_{+})=\eta$ and
\begin{align*}
E_{-} &  \subset\left\{  \frac{d\hat{\beta}}{d\nu}<L-\delta\right\}  \cap
B_{\varepsilon}(y)\\
E_{+} &  \subset\left\{  \frac{d\hat{\beta}}{d\nu}>L+\delta\right\}  \cap
B_{\varepsilon}(y).
\end{align*}
Then let%
\begin{align*}
\hat{\beta}\left(  E_{+}\right)   &  =a\eta\\
\hat{\beta}\left(  E_{-}\right)   &  =b\eta
\end{align*}
so that
\[
a-b\geq2\delta.
\]
Consider the measure $\beta^{\ast}$ created by replacing $\hat{\beta}$ with
the average value of $\hat{\beta}$ on the sets $E_{-}$ and $E_{+}$, namely
\begin{equation}
\frac{d\beta^{\ast}}{d\nu}=%
\begin{cases}
\frac{a+b}{2} & \text{on }E_{-}\cup E_{+}\\
\frac{d\hat{\beta}}{d\nu} & \text{otherwise}%
\end{cases}
\text{ .}\label{combs}%
\end{equation}
Now let $f(u)=ug(u)$ and consider the quantity
\[
\int_{Y}g(\beta)\frac{d\beta}{d\nu}d\nu.
\]
Note that
\begin{equation}
\frac{d^{2}}{du^{2}}f(u)\geq c_{0}(M)>0\text{ for }u\leq M\label{sconvexf}%
\end{equation}
in particular, $f$ is uniformly convex. \ By calculus,
\begin{align*}
f(a)+f(b)-2f\left(  \frac{a+b}{2}\right)   &  \geq\frac{1}{4}\left(
a-b\right)  ^{2}\min_{u\in\lbrack a,b]}f^{\prime\prime}(u)\\
&  \geq c_{0}\delta^{2}.
\end{align*}
Applying Jensen's inequality on the sets $E\pm$
\begin{align}
\int_{E_{-}}f\left(  \frac{d\hat{\beta}}{d\nu}\right)  \frac{d\nu}{\eta}%
+\int_{E_{+}}f\left(  \frac{d\hat{\beta}}{d\nu}\right)  \frac{d\nu}{\eta} &
\geq f\left(  \int_{E_{-}}\left(  \frac{d\hat{\beta}}{d\nu}\right)  \frac
{d\nu}{\eta}\right)  +f\left(  \int_{E_{+}}\left(  \frac{d\hat{\beta}}{d\nu
}\right)  \frac{d\nu}{\eta}\right)  \label{a44}\\
&  =f\left(  \frac{1}{\eta}\int_{E_{-}}d\hat{\beta}\right)  +f\left(  \frac
{1}{\eta}\int_{E_{+}}d\hat{\beta}\right)  \nonumber\\
&  =f\left(  a\right)  +f\left(  b\right)  \nonumber\\
&  \geq2f\left(  \frac{a+b}{2}\right)  +c_{0}\delta^{2}\label{a45}%
\end{align}
Now recall $\nu(E\pm)=\eta$, (\ref{combs})%
\begin{equation}
2f\left(  \frac{a+b}{2}\right)  =\int_{E_{-}\cup E_{+}}f\left(  \frac
{d\beta^{\ast}}{d\nu}\right)  \frac{1}{\eta}d\nu.\label{a46}%
\end{equation}
Multiplying $\eta$ across (\ref{a44}), (\ref{a45}), and (\ref{a46}) and
combining, we see. \ \
\[
\int_{E_{-}}f\left(  \frac{d\hat{\beta}}{d\nu}\right)  d\nu+\int_{E_{+}%
}f\left(  \frac{d\hat{\beta}}{d\nu}\right)  d\nu-\int_{E_{-}\cup E_{+}%
}f\left(  \frac{d\beta^{\ast}}{d\nu}\right)  d\nu\geq\delta^{2}c_{0}(M)\eta,
\]
\ where
\[
M=\left\Vert \frac{d\hat{\beta}}{d\nu}\right\Vert _{L^{\infty}}%
\]
is controlled by Lemma \ref{linfbound}. \ \ Comparing with Claim
\ref{costclaim}, noting that the mass moved is bounded by 2$M\eta,$ and using
the maximality of $\hat{\beta},$ we conclude that
\[
0\geq F(\hat{\beta})-F(\beta^{\ast})\geq\delta^{2}c_{0}(M)\eta-2M\eta
\left\Vert c\right\Vert _{Lip}\text{diam}\left(  B_{\varepsilon}\right)  .
\]
Thus%
\[
\delta^{2}\leq\frac{2M}{c_{0}(M)}\left\Vert c\right\Vert _{Lip}2\varepsilon.
\]
Revisiting our original assumption (\ref{oass}) ,\ we take
\[
\iota=\frac{c_{0}(M)}{5M\left\Vert c\right\Vert _{Lip}}%
\]
and obtain a contradiction. \ \ Immediately we conclude that $\frac
{d\hat{\beta}}{d\nu}$ is $C^{0,1/2}.$
\end{proof}

\bigskip

\subsection{Step 2}

This is essentially \cite[Steps 4 and 5, section 4]{BCMO}. \ In short: by
Brenier's Theorem, we may find $\phi_{N}$ such that $\nabla\phi_{N}\hat{\beta
}=\mu_{N}.$ \ \ Because $\mu_{N}\rightarrow\mu$, the potentials $\phi_{N}$
converge uniformly. \ \ It follows that the Legendre transforms $\hat{\varphi
}_{N}$ converge uniformly as well, to $\hat{\varphi}$. \ If the diameter of
the Laguerre cells does not shrink to zero, there will be a line segment on
which the convex potential $\hat{\varphi}$ is linear. This contradicts
Caffarelli's regularity result \cite{CaffJAMS}, which ensures that the
potential defining the mapping between the absolutely continuous measures
$\mu$ and $\hat{\beta}$ must be strictly convex.

Each $\phi_{N}$ defines a set of Laguerre cells $\left\{  E_{i}\right\}  .$ We
can define a measure:
\[
\hat{\beta}_{N}(Z)=\frac{1}{N}\sum_{i=1}^{N}\frac{\nu(E_{i}\cap Z)}{\nu
(E_{i})}%
\]
which has the property that
\[
\nabla\phi_{N}\hat{\beta}_{N}=\mu_{N}.
\]
The densities $\frac{d\hat{\beta}_{N}}{d\nu}$ will converge pointwise
uniformly to $\frac{d\hat{\beta}}{d\nu},$ because the density $\frac
{d\hat{\beta}_{N}}{d\nu}\left(  y\right)  $ is simply the average of the
continuous density of $\frac{d\hat{\beta}}{d\nu}$ over the cell containing
$y,$ and the cell diameters are shrinking to zero. \ \ \ It follows quickly
that $\hat{F}(\hat{\beta}_{N})\rightarrow\hat{F}\left(  \hat{\beta}\right)  .$

\subsection{\bigskip Step 3}

For each finite player Nash bargaining problem, we choose $\varphi_{N}$
maximizing (\ref{p3}), and obtain a set of Laguerre cells $\left\{
E_{i}\right\}  .$ \ As before, we construct a probability measure on $Y$%
\begin{equation}
\beta_{N}(Z)=\frac{1}{N}\sum_{i=1}^{N}\frac{\nu(E_{i}\cap Z)}{\nu(E_{i}%
)}.\label{betaN}%
\end{equation}
We would like to conclude these $\beta_{N}\rightarrow\beta,$ but we do not yet
have a unique limit $\beta$. \ \ We need to use properties of $\beta_{N}$ to
conclude properties of possible limits $\beta$.

\begin{claim}
\label{13}The measures $\alpha_{N}$ are uniformly mutually absolutely
continuous with respect to $\mu_{N}.$ \ 
\end{claim}

\begin{proof}
Proceeding as in the proof of Proposition \ref{three}, we have
\[
\frac{\int_{_{B_{\varepsilon}}}s(p_{1},y)d\nu}{\int_{E_{1}}s(p_{1},y)d\nu}%
\geq\frac{\int_{_{B_{\varepsilon}}}s(p_{2},y)d\nu}{\int_{E_{2}}s(p_{2},y)d\nu}%
\]
or
\[
\frac{\int_{E_{2}}s(p_{2},y)d\nu}{\int_{E_{1}}s(p_{1},y)d\nu}\geq\frac
{\int_{_{B_{\varepsilon}}}s(p_{2},y)d\nu}{\int_{_{B_{\varepsilon}}}%
s(p_{1},y)d\nu_{1}}\geq\frac{\min s(x,y)}{\max s(x,y)}>0.
\]
Similarly,
\[
\frac{\int_{E_{1}}s(p_{1},y)d\nu}{\int_{E_{2}}s(p_{2},y)d\nu}\geq\frac{\min
s(x,y)}{\max s(x,y)}>0
\]
so also
\[
\frac{\nu(E_{1})}{\nu(E_{2})}\geq\frac{\frac{1}{\max s(x,y)}\int_{Y}%
s(p_{1},y)d\nu_{2}}{\frac{1}{\min s(x,y)}\int_{Y}s(p_{2},y)d\nu_{1}}>\left[
\frac{\min s(x,y)}{\max s(x,y)}\right]  ^{2}=a_{0}>0.
\]
This will be true for any pair, so the ratios of the measures is uniformly
bounded:
\begin{equation}
\frac{a_{0}}{a_{0}+N-1}\leq\alpha(p_{i})\leq\frac{1}{1+(N-1)a_{0}%
}.\label{abbelow}%
\end{equation}

\end{proof}

\begin{claim}
The measures $\alpha_{N}$ have a weak limit $\alpha$ which is absolutely
continuous with respect to $\mu.$ \ \ The density $d\alpha/d\mu$ is bounded
away from zero. \ \ \ \ 
\end{claim}

\begin{proof}
On the compact space $X$, the Wasserstein metric is compact, and equivalent to
weak topology, so there is a weak limit $\alpha.$ \ \ A straightforward
argument in the spirit of Littlewood's principles using (\ref{abbelow}) shows
that for all measurable $E,$
\[
a_{0}\mu(E)\leq\alpha\left(  E\right)  \leq\frac{1}{a_{0}}\mu(E).
\]

\end{proof}

\bigskip Next we consider the convergence of $\varphi_{N}.$ \ In order to
extend these to $X,$ we denote%
\[
\bar{\varphi}_{N}=\max\left\{  \psi\in\mathcal{K}_{Y}:\psi_{|_{P}}\leq
\varphi_{N|_{P}}\right\}  .
\]

\begin{claim}
The functions $\bar{\varphi}_{N}$ converge uniformly on compact subsets of
$X.$
\end{claim}

\begin{proof}
This follows from the fact that $\left(  \mathcal{K}_{Y}\right)  _{0}$ is
compact: $\ $The space of convex functions with subgradients in a bounded set
is compact up to addition of a constant.
\end{proof}

\begin{claim}
Let $\varphi=\lim\bar{\varphi}_{N}.$ \ Then
\[
\left(  \nabla\varphi\right)  _{\#}\alpha=\nu.
\]

\end{claim}

\begin{proof}
We consider the dual problem (cf. \cite[Theorem 5.10 (iii)]{Villani}). \ For
each $N,$ we have
\[
\min_{\substack{\pi\in P(X\times Y)\\\pi_{X}=\alpha_{N}\\\pi_{Y}=\nu}}\int
c(x,y)d\pi=\max_{\varphi\in L^{1}(X,\alpha_{N})}\int_{Y}\varphi^{c}%
(y)d\nu-\int_{X}\varphi(x)d\alpha_{N},
\]
which is realized by
\[
\int c(x,y)d\pi_{N}=\int_{Y}\varphi_{N}^{c}(y)d\nu-\int_{X}\varphi
_{N}(x)d\alpha_{N}.
\]
Now the set of cost-transpose functions $\left\{  \varphi^{c};\varphi
:X\rightarrow%
\mathbb{R}
\right\}  $ is also compact up to addition of a constant. \ Thus the functions
$\varphi_{N}^{c}$ converge uniformly as well. Taking all limits and letting
$\pi$ be the weak limit of $\pi_{N},$ we have
\[
\int c(x,y)d\pi=\int_{Y}\varphi^{c}(y)d\nu-\int_{X}\varphi(x)d\alpha.
\]
By duality, it follows that $\varphi$ describes the optimal map between
$\alpha$ and $\nu.$ \ In particular, for $c=\left\vert x-y\right\vert ^{2}/2$
we conclude that $\left(  \nabla\varphi\right)  _{\#}\alpha=\nu.$
\end{proof}

Now, we simply define
\begin{equation}
\beta:=\left(  \nabla\varphi\right)  _{\#}\mu.\label{betdef}%
\end{equation}
Despite knowing less about the regularity of $\beta$ than we did for
$\hat{\beta}$ in Step 2, we may repeat the essential portion of the argument
found in \cite{BCMO}. Caffarelli's strict convexity result \cite{CaffJAMS}
only requires the densities are bounded and bounded away from zero, which is
true by Claim \ref{13} for $\alpha$ and $\nu.$ We conclude that the potential
$\varphi$ is strictly convex and that the Laguerre cells associated to
$\varphi_{N}$ must have vanishing diameters. \ 

\bigskip

\subsection{\bigskip Step 4}

For either $\varphi_{N}$ or $\hat{\varphi}_{N}$, note that for each cell
$E_{i}$ the following holds for $p_{i}$ and any $y_{i}\in E_{i}:$
\[
\left\vert s(p_{i},y_{i})-\frac{1}{\nu(E_{i})}\int_{E_{i}}s(p_{i}%
,y)d\nu(y)\right\vert \leq\text{osc}_{y\in E_{i}}\text{ }s(p_{i},y).
\]

The average value of the cost over any set must be larger than
\[
b_{0}=e^{-\max_{X\times Y}c(x,y)}.
\]
It follows immediately by the fundamental theorem of calculus that \
\[
\left\vert \ln\left(  s(p_{i},y_{i})-\frac{1}{\nu(E_{i})}\int_{E_{i}}%
s(p_{i},y)d\nu(y)\right)  -\ln s(p_{i},y_{i})\right\vert \leq\frac{1}{b_{0}%
}\text{osc}_{y\in E_{i}}\text{ }s(p_{i},y).
\]
Now for any set of choices of $\left\{  y_{i}\in E_{i}\right\}  $
\begin{equation}
\left\vert \frac{1}{N}\sum_{i=1}^{N}\ln\left(  \frac{1}{\nu(E_{i})}\int%
_{E_{i}}s(p_{i},y)d\nu(y)\right)  -\frac{1}{N}\sum_{i=1}^{N}\ln s(p_{i}%
,y_{i})\right\vert \leq\frac{1}{N}\sum_{i=1}^{N}\frac{1}{b_{0}}\text{osc}%
_{y\in E_{i}}\text{ }s(p_{i},y).\label{thuy}%
\end{equation}
The utility function is continuous, so as the diameters of $E_{i}$ shrink, so
does\ the right hand side of (\ref{thuy}). \ On the other hand
\[
\frac{1}{N}\sum_{i=1}^{N}\ln s(p_{i},y_{i})=-\int_{X}c(x_{i},y_{i})d\mu_{N}.
\]
The cost function is continuous, and the values $y_{i}$ are being chosen from
the subgradient of $\nabla\varphi_{N}(p_{i})$, and $\mu_{N}\rightarrow\mu,$ so
we conclude
\begin{equation}
\lim_{N\rightarrow\infty}\frac{1}{N}\sum_{i=1}^{N}\ln\left(  \frac{1}%
{\nu(E_{i})}\int_{E_{i}}s(p_{i},y)d\nu\right)  =-\int_{X}c(x,\nabla
\varphi(x))d\mu(x).\ \ \label{limcw}%
\end{equation}
This proves the first claim in Step 4.

Next, note that%

\begin{align}
-H_{\nu}(\beta_{N}) &  =-\sum_{i=1}^{N}\int_{E_{i}}\ln\left(  \frac{1}{N}%
/\nu(E_{i})\right)  \frac{1}{N}/\nu(E_{i})d\nu\nonumber\\
&  =-\sum_{i=1}^{N}\ln\left(  \frac{1}{N}/\nu(E_{i})\right)  \frac{1}%
{N}\nonumber\\
&  =\frac{1}{N}\sum_{i=1}^{N}\ln\left(  \frac{\nu(E_{i})}{\mu_{N}\left(
p_{i}\right)  }\right)  .\label{limch}%
\end{align}
Thus, using (\ref{limcw}),(\ref{betdef}), recalling (\ref{Fhatdef}), and
(\ref{p3})%
\begin{align*}
\lim_{N\rightarrow\infty}\left(  \tilde{F}_{N}(\varphi_{N})-\hat{F}(\beta
_{N})\right)   &  =\lim_{N\rightarrow\infty}\left(  \frac{1}{N}\ln\left(
\frac{1}{\nu(E_{i})}\int_{E_{i}}s(p_{i},y)d\nu\right)  +W^{2}(\mu,\beta
_{N})\right)  \\
&  =-\int c(x,\nabla\varphi(x))d\mu+\lim_{N\rightarrow\infty}W^{2}(\mu
,\beta_{N})\\
&  =-W^{2}(\mu,\left(  \nabla\varphi\right)  _{\#}\mu)+W^{2}(\mu,\beta)\\
&  =0.
\end{align*}
This proves (\ref{bbb}). \ A nearly identical argument proves (\ref{bbb2}).

\section{A\ Fourth Order PDE and smoothness of solutions}

\bigskip

In this section we argue that the function $\varphi$ satisfies an elliptic
quasilinear fourth order PDE and enjoys derivative\bigskip\ estimates of all
orders. 

Given a smooth probability measure $\mu$ on $X,$ one can parameterize the
space of probability measures on $Y$ by convex potentials $\varphi$ on $X:$
\ For any $\beta,$ we can solve the optimal transportation problem pairing
$\mu$ with $\beta,$ obtaining $\varphi$ such that
\begin{equation}
\beta=\left(  \nabla\varphi\right)  _{\#}\mu.\label{pYp}%
\end{equation}
On the other hand, for any convex $\varphi$ the subgradient mapping defines a
probability measure via (\ref{pYp}). \ \ \ Using this, we can insert $\varphi$
into the functional $\hat{F}:$%
\begin{align*}
\hat{F}(\varphi) &  =-\int_{Y}\ln\left(  \frac{d\beta}{d\nu}(y)\right)
\frac{d\beta}{d\nu}(y)d\nu(y)-\frac{1}{2}\int_{X}\left\Vert x-\nabla
\varphi(x)\right\Vert ^{2}d\mu(x).\\
&  =-\int_{Y}\ln\left(  \frac{d\beta}{d\nu}(y)\right)  d\beta(y)-\frac{1}%
{2}\int_{X}\left\Vert x-\nabla\varphi(x)\right\Vert ^{2}d\mu(x).
\end{align*}
Now $\nabla\varphi$ is a change of measure, so%
\begin{align*}
\hat{F}(\varphi) &  =-\int_{X}\ln\left(  \frac{d\beta}{d\nu}\left(
\nabla\varphi(x)\right)  \right)  d\mu(x)-\frac{1}{2}\int_{X}\left\Vert
x-\nabla\varphi(x)\right\Vert ^{2}d\mu(x)\\
&  =-\int_{X}\ln\left(  \frac{\mu(x)}{\det(D^{2}\varphi(x))\nu(\nabla
\varphi(x))}\right)  d\mu(x)-\frac{1}{2}\int_{X}\left\Vert x-\nabla
\varphi(x)\right\Vert ^{2}d\mu(x).\\
&  =-\int_{X}\left[  \ln\mu(x)-\ln\det(D^{2}\varphi(x))-\ln\nu(\nabla
\varphi(x))\right]  d\mu(x)-\frac{1}{2}\int_{X}\left\Vert x-\nabla
\varphi(x)\right\Vert ^{2}d\mu(x).
\end{align*}
Now consider a compactly supported variation:
\[
\varphi_{t}=\varphi(x)+t\eta(x)
\]
for some compactly supported smooth test function $\eta.$ \ \ We compute
\begin{align*}
\frac{d\hat{F}(\varphi_{t})}{dt}|_{t=0} &  =-\int_{X}\left[  -\varphi^{ij}%
\eta_{ij}-\frac{1}{\nu(\nabla\varphi)}\nabla\nu(\nabla\varphi)\cdot\nabla
\eta\right]  d\mu-\int_{X}\left(  x-\nabla\varphi\right)  \cdot\nabla\eta
d\mu\\
&  =\int_{X}\left\{  \partial_{i}\partial_{j}\left(  \mu\varphi^{ij}\right)
+\operatorname{div}\left(  -\mu\frac{1}{\nu(\nabla\varphi)}\nabla\nu
(\nabla\varphi)+\mu\left(  x-\nabla\varphi\right)  \right)  \right\}  \eta dx.
\end{align*}
Here $\varphi^{ij}$ is the inverse of the Hessian matrix $\varphi_{ij}.$ \ Now
if the measure $\beta=\left(  \nabla\varphi\right)  _{\#}\mu$ is a maximizer,
any compactly supported variation will not change the functional to first
order, so we have an Euler-Lagrange equation:
\[
\partial_{i}\partial_{j}\left(  \mu\varphi^{ij}\right)  +\operatorname{div}%
\left(  \mu\frac{1}{\nu(\nabla\varphi(x))}\nabla\nu(\nabla\varphi
(x))+\mu\left(  x-\nabla\varphi\right)  \right)  =0.
\]
This becomes an equation on $\det(D^{2}\varphi(x)):$ \ Write the first term as%

\[
\partial_{i}\partial_{j}\left(  \mu\varphi^{ij}\right)  =\partial_{i}%
\partial_{j}\left(  \mu\frac{C_{\varphi}^{ij}}{\det D^{2}\varphi(x)}\right)
\]
where%
\[
C_{\varphi}^{ij}=\det D^{2}\varphi(x)\varphi^{ij}%
\]
is the (divergence-free) cofactor matrix. \ \ We see%
\[
\partial_{i}\partial_{j}\left(  \mu\varphi^{ij}\right)  =L\left(  \frac
{\mu(x)}{\det D^{2}\varphi(x)}\right)
\]
where
\[
L=C^{ij}\partial_{i}\partial_{j}%
\]
the equation becomes
\begin{equation}
L\left(  \frac{\mu(x)}{\det D^{2}\varphi(x)}\right)  =\operatorname{div}%
\left(  \mu(x)\frac{1}{\nu(\nabla\varphi(x))}\nabla\nu(\nabla\varphi
(x))-\mu(x)\left(  x-\nabla\varphi(x)\right)  \right)  .\label{qlin}%
\end{equation}
Now we determined in Lemma \ref{holder} that the density $\frac{d\hat{\beta}%
}{d\nu}$ must be H\"{o}lder continuous. \ It follows that the potential
$\varphi$ satisfying
\[
\det(D^{2}\varphi(x))=\frac{\mu(x)}{\hat{\beta}(\nabla\varphi(x))}%
\]
will be $C^{2,\alpha}$, $\ $with estimates on any interior set \cite{CaffJAMS}%
. \ In particular, the cofactor matrix defining $L$ is uniformly elliptic.
\ Thus the equation (\ref{qlin}) is uniformly elliptic with  H\"{o}lder
coefficients. Also note that because $\varphi\in C^{2,\alpha}$, the right hand
side of (\ref{qlin}) is $C^{0,\alpha}.$ \ Thus we can apply the classical
Schauder theory \cite[Theorem 6.19]{GT}  and conclude that $\det D^{2}%
\varphi(x)$ is itself $C^{2,\alpha}.$ \ \ Repeating these two steps gives
estimates of arbitrarily high order.

\bibliographystyle{plain}
\bibliography{nash}

\end{document}